\documentclass[10pt]{amsart}

\usepackage{amsfonts,amsmath,amsthm,amscd,amssymb,latexsym,mathrsfs}
\usepackage[english]{babel}

\newtheorem{thm}{Theorem}[section]

\newtheorem{lem}[thm]{Lemma}
\newtheorem{prop}[thm]{Proposition}

\newtheorem{exm}[thm]{Example}

\theoremstyle{definition}
\newtheorem{defn}[thm]{Definition}
\theoremstyle{remark}
\newtheorem{rem}[thm]{Remark}
\numberwithin{equation}{section}

\begin{document}

\title{On stable Baire classes}
\author{Olena Karlova}\email{maslenizza.ua@gmail.com}
\author{Volodymyr Mykhaylyuk}\email{vmykhaylyuk@ukr.net}
\address{Yurii Fedkovych Chernivtsi National University, Ukraine}



\date{}

\begin{abstract}
 We introduce and study adhesive spaces. Using this concept we obtain  a characterization of stable Baire maps $f:X\to Y$ of the class $\alpha$ for wide classes of topological spaces. In particular, we prove that for a topological space  $X$ and a contractible space  $Y$ a map $f:X\to Y$ belongs to the $n$'th stable Baire class if and only if there exist a sequence  $(f_k)_{k=1}^\infty$ of continuous maps  $f_k:X\to Y$ and a sequence $(F_k)_{k=1}^\infty$ of functionally ambiguous sets  of the $n$'th class in $X$ such that  $f|_{F_k}=f_k|_{F_k}$ for every $k$. Moreover, we show that every monotone function  $f:\mathbb R\to \mathbb R$ is of the $\alpha$'th stable Baire class if and only if it belongs to the first stable Baire class.
\end{abstract}

\maketitle

\section{Introduction, terminology and notations}

We say that a sequence $(f_n)_{n=1}^\infty$ of maps $f_n:X\to Y$ between topological spaces $X$ and $Y$ is  {\it stably convergent to a map $f:X\to Y$ on $X$}, if for every $x\in X$ there exists $k\in\mathbb N$ such that $f_n(x)=f(x)$ for all  $n\ge k$. A map $f:X\to Y$ belongs to {\it the first stable Baire class}, if there exists a sequence of continuous maps between $X$ and $Y$ which is stably convergent to $f$ on $X$.

Real-valued functions of the first stable Baire class on a topological space $X$ naturally appear both as an interesting subclass of all differences of semi-continuous functions on $X$ \cite{ChRos,HOR} and in problems on the Baire classification of integrals depending on a parameter \cite{BKMM} as well as in problems concerning a composition of Baire functions \cite{Karlova:Mykhaylyuk:Comp}.

Real-valued functions of higher stable Baire classes were introduced and studied by \'{A}.~Cs\'{a}sz\'{a}r and M.~Laczkovich in \cite{CsLacz:1,CsLacz:2}. A characterization of maps of the first Baire class defined on a perfectly paracompact hereditarily Baire space with the Preiss-Simon property and with values in  a path-connected space with special extension properties was established by T.~Banakh and B.~Bokalo in \cite{BaBo}.

This paper is devoted to obtain a characterization of stable Baire maps for  wide classes of topological spaces and any ordinal $\alpha\in[1,\omega_1)$.
To do this we introduce a class of adhesive spaces and study their properties in Section~\ref{sec:AdhesiveSpaces}. In Section~\ref{sec:Stable} we give a characterization of stable Baire maps defined on a topological space and with values in adhesive spaces (see Theorem~\ref{th:char_B1st}). Finally, in Section~\ref{sec:monotone} we apply this characterization to classify monotone functions within stable Baire classes (see Theorem~\ref{th:mon_stable}).

Let  us give some notations and recall several definitions. For  topological spaces $X$ and $Y$ by ${\rm C}(X,Y)$ we denote the collection of all continuous maps between $X$ and $Y$.

If $A\subseteq Y^X$, then the symbol $\overline{A}^{\,\,{\rm st}}$ stands for the collection of all stable   limits of sequences of maps from~$A$.
We put
$$
{\rm B}_0^{{\rm st}}(X,Y)={\rm C}(X,Y)
$$
and for each ordinal $\alpha\in (0,\omega_1)$ let ${\rm B}_\alpha^{\rm st}(X,Y)$ be the family of all maps of {\it the $\alpha$'th stable Baire class} which is defined by the formula
$$
{\rm B}_\alpha^{\rm st}(X,Y)=\overline{\bigcup\limits_{\beta<\alpha}{\rm B}_\beta^{\rm st}(X,Y)}^{\,\,{\rm st}}.
$$

Recall that a set $A\subseteq X$ is {\it functionally closed}, if there exists a continuous function $f:X\to [0,1]$ with $A=f^{-1}(0)$. If the complement of $A$ is functionally closed, then $A$ is called {\it functionally open}.

Let $\mathcal M_0(X)$ be the family of all functionally closed subsets of  $X$ and let $\mathcal A_0(X)$ be the family of all functionally open subsets of  $X$. For every  $\alpha\in [1,\omega_1)$ we put
\begin{gather*}
   \mathcal M_{\alpha}(X)=\Bigl\{\bigcap\limits_{n=1}^\infty A_n: A_n\in\bigcup\limits_{\beta<\alpha}\mathcal A_{\beta}(X),\,\, n=1,2,\dots\Bigr\}\,\,\,\mbox{and}\\
   \mathcal A_{\alpha}(X)=\Bigl\{\bigcup\limits_{n=1}^\infty A_n: A_n\in\bigcup\limits_{\beta<\alpha}\mathcal M_{\beta}(X),\,\, n=1,2,\dots\Bigr\}.
 \end{gather*}
Elements from the class  $\mathcal M_\alpha(X)$ belong to {\it the $\alpha$'th functionally multiplicative class} and elements from  $\mathcal A_\alpha(X)$ belong to {\it the $\alpha$'th functionally additive class} in  $X$. We say that a set is {\it functionally ambiguous of the $\alpha$'th class}, if it belongs to $\mathcal M_\alpha(X)$ and $\mathcal A_\alpha(X)$ simultaneously.

A topological space $X$ is called {\it contractible}, if there exist a point $x_0\in X$ and a continuous map $\gamma:X\times[0,1]\to X$ such that $\gamma(x,0)=x$ and $\gamma(x,1)=x_0$ for all $x\in X$.
 A space $Y$ is {\it an extensor for $X$}, if for any closed set $F\subseteq X$ each continuous map $f:F\to Y$  can be extended to a continuous map $g:X\to Y$.

A map $f:X\to Y$ is called {\it piecewise continuous}, if there exists a cover $(F_n:n\in\mathbb N)$ of  $X$ by closed sets such that each
restriction $f|_{F_n}$ is continuous.

We denote by symbols $C(f)$ and $D(f)$ the sets of all points of continuity and discontinuity of a map $f:X\to Y$, respectively.

\section{Adhesive spaces}\label{sec:AdhesiveSpaces}


\begin{defn}\label{Def:Adhesive}
  We say that a topological space $Y$ is  {\it an adhesive for $X$} (we denote this fact by  $Y\in {\rm Ad}(X)$), if for any two disjoint functionally closed  sets   $A$ and  $B$ in $X$ and any continuous maps $f,g:X\to Y$ there exists a continuous map $h:X\to Y$ such that $h|_A=f|_A$ and $h|_B=g|_B$. A space $Y$ is said to be  {\it an absolute adhesive} for a class $\mathcal C$ of topological spaces and is denoted by $Y\in {\rm AAd(\mathcal C)}$, if $Y\in{\rm Ad}(X)$ for any  $X\in\mathcal C$.
\end{defn}

If the case when $Y$ is an adhesive for any space  $X$, then we will write  $Y\in {\rm AAd}$.

\begin{rem}
\begin{enumerate}
  \item Every extensor is an adhesive.

  \item Let $\pi(X,Y)$ be the set of all homotopy classes of continuous maps between $X$ and $Y$. If $\pi(X,Y)={\rm C}(X,Y)$, then $Y$ is an adhesive for   $X$.

  \item Example~\ref{exm:cantorAD}(a) shows that the class of adhesive spaces is strictly wider than the class of extensors. Example~\ref{exm:cantorAD}(b) contains an adhesive $Y$ for $X$ such that $\pi(X,Y)\ne{\rm C}(X,Y)$.
\end{enumerate}
\end{rem}

\begin{defn}\label{Def:StarAdhesive}
Let $E$ be a subspace of a topological space $Y$. The pair  $(Y,E)$ is called   {\it \mbox{a $*$-adhesive} for  $X$} (and is denoted by $(Y,E)\in{\rm Ad}^*(X)$), if there exists a point   $y^*\in E$ such that for any continuous map $f:X\to Y$ and any two disjoint functionally closed sets   $A$ and $B$ in  $X$ with $f(A)\subseteq E$ there exists a continuous map  $h:X\to Y$ such that $h|_A=f|_A$ and $h|_B=y^*$.
\end{defn}

\begin{rem} A pair $(Y,E)$ is  a $*$-adhesive for $X$, if
  \begin{enumerate}
    \item $E$ is a subspace of  $Y\in {\rm Ad}(X)$;

    \item $(Y,E)\in{\rm AE}(X)$ (i.e., each continuous map $f:F\to E$ has a continuous extension \mbox{$g:X\to Y$});

    \item  $E\in {\rm Ad}(X)$ and $E$ is a retract of  $Y$.
  \end{enumerate}
\end{rem}

\begin{defn}
  A topological space  $Y$ is said to be {\it a $\sigma$-adhesive for $X$}, if there exists a cover $(Y_n:n\in\mathbb N)$ of $Y$ by functionally closed subspaces  $Y_n$ such that $(Y,Y_n)\in {\rm Ad}^*(X)$ for every $n$.
\end{defn}

\begin{defn}
 A topological space $X$ is {\it low-dimensional}, if each point of $X$ has a base of open neighborhoods with discrete   (probably, empty) boundaries.
\end{defn}

It is clear that for a regular low-dimensional space we have ${\rm ind}X\le 1$.

The following result gives examples of adhesive spaces.

\begin{thm}\label{Thm:AbsAdh}
Let $Y$ be a topological space. Then
  \begin{enumerate}
    \item\label{Thm:AbsAdh:it:1} $Y$ is an absolute adhesive for the class of all strongly zero-dimensional spaces;

    \item\label{Thm:AbsAdh:it:2}  $Y$ is an absolute adhesive for the class of all low-dimensional compact Hausdorff spaces, if  $Y$ is path-connected;

    \item\label{Thm:AbsAdh:it:3}  $Y\in {\rm AAd}$ if and only if  $Y$ is contractible.
  \end{enumerate}
\end{thm}

\begin{proof} Let $X$ be a topological space, $A,B$ be disjoint functionally closed subsets of $X$ and  $f,g:X\to Y$ be continuous maps.

{\bf (\ref{Thm:AbsAdh:it:1}).} Assume that $X$ is strongly zero-dimensional and choose a clopen set  $U\subseteq X$ such that $A\subseteq U\subseteq X\setminus B$. Then the map $h:X\to Y$,
$$
h(x)=\left\{\begin{array}{ll}
              f(x), & x\in U, \\
              g(x), & x\in X\setminus U,
            \end{array}
\right.
$$
is continuous, $h|_A=f|_A$ and $h|_B=g|_B$.

{\bf (\ref{Thm:AbsAdh:it:2}).} Assume that $X$ is a low-dimensional Hausdorff compact space and $Y$ is a path-connected space. We choose a continuous function $\varphi:X\to [0,1]$ such that $A=\varphi^{-1}(0)$ and $B=\varphi^{-1}(1)$. For each point $x\in A\cup B$ we take its open neighborhood   $O_x$ with the discrete boundary $\partial O_x$ such that  $O_x\subseteq \varphi^{-1}([0,1/3))$ for $x\in A$ and $O_x\subseteq \varphi^{-1}((2/3,1])$ for $x\in B$. Since   $X$ is compact, every boundary $\partial O_x$ is finite. Moreover, there exist finite subcovers $\mathcal U$ and $\mathcal V$ of $(O_x:x\in A)$ and $(O_x:x\in B)$, respectively. Then the sets $U=\cup\mathcal U$ and $V=\cup\mathcal V$ are open neighborhoods of $A$ and $B$, respectively, $\partial U\cap\partial V=\emptyset$ and $|\partial U\cup\partial V|<\aleph_0$. In the case when one of the boundaries of the sets $U$ or $V$ is empty, we can construct the required map   $h$ as in case (\ref{Thm:AbsAdh:it:1})). Hence, we may suppose that  $\partial U\ne\emptyset\ne\partial V$. Let $\partial U=\{x_1,\dots,x_n\}$ and $\partial V=\{x_{n+1},\dots,x_{n+m}\}$ for some $n,m\in\mathbb N$. Taking into account that the space $D=\partial U\cup\partial V$ is finite and Hausdorff, we obtain that a function $\psi:D\to [1,n+m]$, which is defined by the equality $\psi(x_i)=i$ for $i\in\{1,\dots,n+m\}$, is continuous. Let $\tilde\psi:X\to [1,n+m]$ be a continuous  extension of   $\psi$. Denote $y_i=f(x_i)$ for $i\in\{1,\dots,n\}$ and  $y_i=g(x_i)$ for $i\in\{n+1,\dots,n+m\}$. Now we use the path-connectedness of  $Y$ and for every  $i\in\{1,\dots,n+m-1\}$ find a continuous map $\gamma_i:[i,i+1]\to Y$ such that $\gamma_i(i)=y_i$ and $\gamma_i(i+1)=y_{i+1}$. The maps $\gamma_i$ compose a single continuous map $\gamma:[1,n+m]\to Y$ such that $\gamma|_{[i,i+1]}=\gamma_i$ for all $i\in\{1,\dots,n+m-1\}$. We define a map  $h:X\to Y$,
$$
h(x)=\left\{\begin{array}{ll}
              \gamma(\tilde\psi(x)), & x\in X\setminus (U\cup V), \\
              f(x),& x\in U,\\
              g(x),& x\in V.
            \end{array}
\right.
$$
Notice that $h$ is continuous,  $h|_A=f|_A$ and $h|_B=g|_B$.

{\bf (\ref{Thm:AbsAdh:it:3}).} {\it Necessity.} Fix $y^*\in Y$. Consider the space  $X=Y\times [0,1]$, its disjoint functionally closed subsets  $A=Y\times\{0\}$ and $B=Y\times\{1\}$, and continuous maps $f,g:X\to Y$ such that $f(y,t)=y$ and $g(y,t)=y^*$ for all $y\in Y$ and  $t\in [0,1]$. Since $Y$ is adhesive for $X$, there exists a continuous map $\lambda:X\to Y$ such that $\lambda|_A=f|_A$ and $\lambda|_B=g|_B$. This implies that $Y$ is contractible.

{\it Sufficiency.} Assume that $Y$ is contractible and let  $\lambda:Y\times[0,1]\to Y$ be a continuous map such that $\lambda(y,0)=y$ and $\lambda(y,1)=y_0$ for all $y\in Y$. Consider a topological space $X$, its disjoint functionally closed subsets   $A$ and $B$, and continuous maps $f,g:X\to Y$. Let $\varphi:X\to [0,1]$ be a continuous function such that $A=\varphi^{-1}(0)$ and $B=\varphi^{-1}(1)$. Then a map   $h:X\to Y$,
$$
h(x)=\left\{\begin{array}{ll}
              \lambda(f(x),2\varphi(x)), & \varphi(x)\in [0,\tfrac 12], \\
              \lambda(g(x),2-2\varphi(x)), & \varphi(x)\in (\tfrac 12,1],
            \end{array}
\right.
$$
is continuous and satisfies conditions from Definition~\ref{Def:Adhesive}.
\end{proof}

The following example indicates the essentiality of path-connectedness in Theorem~\ref{Thm:AbsAdh}~(\ref{Thm:AbsAdh:it:2}).

\begin{exm}
  There exists a connected set $Y\subseteq\mathbb R^2$ such that for any two distinct continuous functions $f,g:[0,1]\to Y$ and any two disjoint closed sets $A$ and $B$ in $[0,1]$ there is no continuous function $h:[0,1]\to Y$ with $h|_A=f|_A$ and $h|_B=g|_B$.
\end{exm}

\begin{proof} Let $\mathbb Q=\{r_n:n\in\mathbb N\}$ be the set of all rational numbers. For every $n\in\mathbb N$ we consider the function $\varphi_n:\mathbb R\to\mathbb R$,
  $$
  \varphi_n(x)=\left\{\begin{array}{ll}
                        \sin\frac{1}{x-r_n}, & x\ne r_n, \\
                        0, & x=r_n.
                      \end{array}
  \right.
  $$
Define the function $\varphi:\mathbb R\to\mathbb R$,
  $$
  \varphi(x)=\sum\limits_{n=1}^\infty \frac{1}{2^n}\varphi_n(x).
  $$
Let
$$
  Y=\{(x,y)\in\mathbb R^2:y=\varphi(x)\}.
$$
   Observe that for every  $n$ the function $\psi_n(x)=\sum\limits_{k=1}^n \frac{1}{2^k}\varphi_k(x)$ is a Baire-one Darboux function. Since the sequence $(\psi_n)_{n=1}^\infty$ is uniformly convergent to $\varphi$ on $\mathbb R$, $\varphi$ is a Baire-one Darboux function~\cite[Theorem 3.4]{Bru}.  Consequently, the graph $Y$ of $\varphi$ is connected  according to \cite[Theorem 1.1]{Bru}. Notice that the space $Y$ is punctiform (i.e., $Y$ does not contain any continuum of the cardinality larger than one), since $\varphi$ is discontinuous on everywhere dense set $\mathbb Q$ (see \cite{KuSerp}). Then each continuous function between $\mathbb R$ and $Y$ is constant. The statement of the example follows immediately.
   \end{proof}

\begin{exm}\label{exm:cantorAD}{\rm
 a) Let $X=[0,1]$, $C\subseteq [0,1]$ be the Cantor set and $Y=\Delta (C)\subseteq\mathbb R^2$ be the cone over $C$, i.e., the collection of all segments which connect the point $v=(\tfrac 12,1)$ with points $(x,0)$ for all $x\in C$. Then, being contractible, $Y\in {\rm AAd}$. We show that $Y$ is not an extensor for $X$. Indeed, denote by $((a_n,b_n))_{n=1}^\infty$ the sequence of contiguous intervals to the Cantor set and consider the identity embedding  $f:C\to C\times\{0\}$. Assume that there exists a continuous extension $g:[0,1]\to Y$ of the function $f$. Then $g([a_n,b_n])\supseteq [a_n,v]\cup[b_n,v]$ for every   $n\in\mathbb N$, which implies the equality $g([0,1])=Y$. Therefore, we obtain a contradiction, since $Y$ is not locally path-connected.

 b) Let $X=Y=S^1=\{(x,y)\in\mathbb R^2: x^2+y^2=1\}$. By Theorem~\ref{Thm:AbsAdh} the space $S^1$ is adhesive for itself. On the other hand, the continuous maps $f,g:S^1\to S^1$  defined by the equalities $f(x,y)=(x,y)$ and  $g(x,y)=(1,0)$, are not homotopic.}
\end{exm}

\section{Stable Baire classes and their characterization}\label{sec:Stable}

We omit the proof of the following fact, since it is completely similar to the proof of Theorem 2 from  \cite[p.~357]{Kuratowski:Top:1}.

\begin{lem}\label{amb}
 If  $A$ is a functionally ambiguous set of the class  $\alpha>1$ in a topological space $X$, then there exists a sequence $(A_n)_{n=1}^\infty$ of functionally ambiguous sets of classes  $<\alpha$ such that
  \begin{gather}\label{gath:limitset}
    A=\mathop{\rm Lim}\limits_{n\to\infty} A_n=\bigcup\limits_{n=1}^\infty \bigcap\limits_{k=0}^\infty A_{n+k}=\bigcap\limits_{n=1}^\infty \bigcup\limits_{k=0}^\infty A_{n+k}.
  \end{gather}
  Moreover, if $\alpha=\beta+1$ for a limit ordinal  $\beta$, then all the sets $A_n$ can be taken from classes  $<\beta$.
\end{lem}

\begin{thm}\label{th:char_B1st}
Let  $X$ be a topological space, $Y$ be a topological space with  the functionally closed diagonal $\Delta=\{(y,y):y\in Y\}$, $\alpha\in[1,\omega_1)$ and let  $\beta=\alpha$, if  $\alpha<\omega_0$,  and  $\beta=\alpha+1$, if  $\alpha\ge\omega_0$.
For a map $f:X\to Y$ we consider the following conditions:
   \begin{enumerate}
    \item\label{it:th:char_B1st:1} $f\in {\rm B}_\alpha^{\rm st}(X,Y)$;

    \item\label{it:th:char_B1st:2} there exist an increasing sequence $(X_n)_{n=1}^\infty$ of sets of functionally multiplicative classes  $<\beta$ and a sequence $(f_n)_{n=1}^\infty$ of maps $f_n\in{\rm B}_{<\alpha}^{\rm st}(X,Y)$ such that $X=\bigcup\limits_{n=1}^\infty X_n$ і $f_n|_{X_n}=f|_{X_n}$ for all $n\in\mathbb N$;

    \item\label{it:th:char_B1st:3} there exist a partition $(X_n:n\in\mathbb N)$ of $X$ by functionally ambiguous sets of the class $\beta$ and a sequence of continuous maps $f_n:X\to Y$ such that  $f_n|_{X_n}=f|_{X_n}$ for all $n\in\mathbb N$.
  \end{enumerate}
  Then  $(\ref{it:th:char_B1st:1})\Leftrightarrow(\ref{it:th:char_B1st:2})\Rightarrow (\ref{it:th:char_B1st:3})$. If one of the following conditions hold
  \begin{enumerate}
   \item[(a)] $Y$ is adhesive for $X$, or

   \item[(b)] $Y$ is path-connected $\sigma$-adhesive for  $X$,
  \end{enumerate}
then $(\ref{it:th:char_B1st:3})\Rightarrow(\ref{it:th:char_B1st:2})$.
\end{thm}

\begin{proof}   {\bf (\ref{it:th:char_B1st:1})$\Rightarrow$(\ref{it:th:char_B1st:2}).} Assume that the diagonal $\Delta$ is functionally closed in $Y^2$. Let  $(f_n)_{n=1}^\infty$ be a sequence of maps $f_n\in{\rm B}_{<\alpha}^{\rm st}(X,Y)$ which is stably convergent to  $f$ on $X$. For $k,n\in\mathbb N$ we put
  \begin{gather*}
X_{k,n}=\{x\in X:f_k(x)=f_n(x)\}\quad\mbox{and}\quad  X_n=\bigcap\limits_{k=n}^\infty X_{k,n}.
  \end{gather*}
Clearly, $X_n\subseteq X_{n+1}$, $X=\bigcup\limits_{n=1}^\infty X_n$ and $f_n|_{X_n}=f|_{X_n}$ for every $n\in\mathbb N$.

For all $x\in X$ and $k,n\in\mathbb N$ we put $h_{k,n}(x)=(f_k(x),f_n(x))$. In the case $\alpha=\gamma+1<\omega_0$ we can assume that $f_n\in {\rm B}_{\gamma}^{\rm st}(X,Y)$ for all $n\in\mathbb N$. Then the equality
$$
X_{k,n}=h_{k,n}^{-1}(\Delta)
$$
implies that $X_{k,n}\in \mathcal M_{\gamma}$ and $X_n\in \mathcal M_\gamma$. Suppose    $\alpha\ge\omega_0$. If $\alpha=\omega_0$, then each map $f_n$ can be taken from the class ${\rm B}_{n}^{\rm st}(X,Y)$. Then $X_{k,n}\in\mathcal M_{k}$ for all  $k\ge n$, which implies that $X_n\in\mathcal M_{\omega_0}$. Now let  $\alpha>\omega_0$. We can assume that $f_n\in {\rm B}_{\alpha_n}^{\rm st}(X,Y)$, where $\omega_0\le\alpha_n<\alpha_n+1\le\alpha$ for all $n\in\mathbb N$. Then $X_{k,n}\in\mathcal M_{\max\{\alpha_n,\alpha_k\}+1}\subseteq\mathcal M_\alpha$ and $X_n\in\mathcal M_\alpha$.

  {\bf (\ref{it:th:char_B1st:2})$\Rightarrow$(\ref{it:th:char_B1st:1}).} Since the sequence $(X_n)_{n=1}^\infty$ is increasing, $(f_n)_{n=1}^\infty$ is convergent stably to $f$ on $X$.

 {\bf (\ref{it:th:char_B1st:2})$\Rightarrow$(\ref{it:th:char_B1st:3}).} We will argue by the induction.  For $\alpha=1$ we take a sequence $(F_n)_{n=1}^\infty$ of functionally closed sets and a sequence $(f_n)_{n=1}^\infty$ of continuous maps  $f_n:X\to Y$ such that $f_n|_{F_n}=f|_{F_n}$ and $X=\bigcup\limits_{n=1}^\infty F_n$. We set $X_1=F_1$ and $X_{n}=F_n\setminus (F_1\cup\dots\cup F_{n-1})$  for $n\ge 2$. Then the family $(X_n:n\in\mathbb N)$ is the required partition of the space $X$.

Assume that the implication (\ref{it:th:char_B1st:2})$\Rightarrow$(\ref{it:th:char_B1st:3}) holds for all $\gamma<\alpha$ for some $\alpha\in[1,\omega_0)$. Let $(A_n)_{n=1}^\infty$ be an increasing sequence of sets of the $(\alpha-1)$'th functionally multiplicative class and let  $(g_n)_{n=1}^\infty$ be a sequence of maps from the class ${\rm B}_{\alpha-1}^{\rm st}(X,Y)$ such that
\begin{gather}\label{gath:th:char_B1st:1}
X=\bigcup\limits_{n=1}^\infty A_n\quad\mbox{and}\quad g_n|_{A_n}=f|_{A_n}\quad\mbox{for every\,\,} n\in\mathbb N
\end{gather}
By the inductive assumption and by equivalence $(\ref{it:th:char_B1st:1})\Leftrightarrow(\ref{it:th:char_B1st:2})$, for every $n$ there exist  a sequence $(B_{nm})_{m=1}^\infty$ of mutually disjoint functionally ambiguous sets of the class  $\alpha-1$ and a sequence $(h_{nm})_{m=1}^\infty$ of continuous maps $h_{nm}:X\to Y$ such that
\begin{gather}\label{gath:th:char_B1st:2}
  h_{nm}|_{B_{nm}}=g_n|_{B_{nm}}\quad\mbox{for all}\quad m\in\mathbb N.
\end{gather}
For all $n,m\in\mathbb N$ we set
 \begin{gather}\label{gath:th:char_B1st:3}
X_{nm}=(A_{n}\setminus \bigcup\limits_{k=0}^{n-1} A_{k})\cap B_{nm},
 \end{gather}
 where $A_0=\emptyset$. Then the partition  $(X_{nm}:n,m\in\mathbb N)$ of the space $X$ is the required one.

We show that (\ref{it:th:char_B1st:2})$\Rightarrow$(\ref{it:th:char_B1st:3}) for all $\alpha\in [\omega_0,\omega_1)$ under the assumption that $Y$ has the functionally closed diagonal. Again we will argue by the transfinite induction. Let $\alpha=\omega_0$, $(A_n)_{n=1}^\infty$ be an increasing sequence of sets of the $\omega_0$'th functionally multiplicative class and $(g_n)_{n=1}^\infty$ be a sequence of maps $g_n\in{\rm B}_n^{\rm st\,\,}(X,Y)$ such that (\ref{gath:th:char_B1st:1}) holds. By implication  (\ref{it:th:char_B1st:1})$\Rightarrow$(\ref{it:th:char_B1st:2}) proved above  for  $g_n$ and by implication (\ref{it:th:char_B1st:2})$\Rightarrow$(\ref{it:th:char_B1st:3}) proved above for finite ordinals, we obtain that for every $n\in\mathbb N$ there exist a partition $(B_{nm})_{m=1}^\infty$ of the space $X$ by functionally ambiguous sets of the class $n$ and a sequence $(h_{nm})_{m=1}^\infty$ of continuous maps such that (\ref{gath:th:char_B1st:2}) is valid. It remains to apply (\ref{gath:th:char_B1st:3}). Further, the inductive step is proved similarly to the case of finite ordinals.

Now we prove that {\bf (\ref{it:th:char_B1st:3})$\Rightarrow$(\ref{it:th:char_B1st:2})} in the case $\alpha=1$.
Assume that condition (a) holds. For every  $n\in\mathbb N$ we take an increasing sequence $(F_{nm})_{m=1}^\infty$ of functionally closed sets in $X$ such that   $X_n=\bigcup\limits_{m=1}^\infty F_{nm}$ and set $\tilde X_n=\bigcup\limits_{m=1}^n F_{mn}$. Then  $(\tilde X_n)_{n=1}^\infty$ is an increasing sequence of functionally closed sets which covers the space $X$. Since $Y$ is adhesive for $X$, for every $n\in\mathbb N$ there exists a continuous map $\tilde f_n:X\to Y$ such that $\tilde f_n|_{F_{mn}}=f_m|_{F_{mn}}$ for all $m\in\{1,\dots,n\}$.  Clearly,  $\tilde f_n|_{\tilde X_n}=f|_{\tilde X_n}$ for all $n\in\mathbb N$.

Now suppose that condition {\bf (b)} holds. Let $Y$ be a path-connected $\sigma$-adhesive for $X$ and $(Y_n:n\in\mathbb N)$ be a cover of the space $Y$
by functionally closed subspaces  $Y_n$ such that $(Y,Y_n)\in {\rm Ad}^*(X)$ for every  $n$.

We prove that the preimage of each functionally closed subset of $Y$ under the mapping $f$ is functionally ambiguous of the class $\beta$ in $X$.
Indeed, take a functionally closed set $B\subseteq Y$. Then $f^{-1}(B)=\bigcup\limits_{n=1}^\infty (f_n^{-1}(B)\cap X_n)$. Since $f_n:X\to Y$ is continuous, $f_n^{-1}(B)$ is functionally closed in $X$. Therefore, $f^{-1}(B)$ belongs to the $\beta$'th additive class in $X$. Moreover,
$X\setminus f^{-1}(B)=\bigcup\limits_{n=1}^\infty (f_n^{-1}(Y\setminus B)\cap X_n)$. Since $f_n^{-1}(Y\setminus B)$ is functionally open in $X$, we have that $X\setminus f^{-1}(B)$ belongs to the functionally additive class $\beta$ in $X$. Thus, $f^{-1}(B)$ is functionally ambiguous of the $\beta$'th class in $X$.


For every $k,n\in\mathbb N$ we put $X_{k,n}=f^{-1}(Y_k)\cap X_n$. Let us remove from the  sequence $(X_{k,n})_{k,n=1}^\infty$ empty sets and let $(Z_n)_{n=1}^\infty$ be an enumeration of the double sequence. Denote $\tilde X_1=Z_1$ and $\tilde X_n=Z_n\setminus \bigcup\limits_{k<n}Z_k$ for $n>1$.
Notice that $(\tilde X_n:n\in\mathbb N)$ is a partition of $X$ by functionally ambiguous sets of the class $\beta$ in $X$. For every $k\in\mathbb N$ we set $N_k=\{n\in\mathbb N: f(\tilde X_n)\subseteq Y_k\}$ and put $\tilde Y_i=Y_k$ for all $i\in N_k$. Hence, we obtain a partition $(\tilde X_n:n\in\mathbb N)$ of $X$ by functionally ambiguous sets of the $\beta$'th class in $X$ such that $f(\tilde X_n)\subseteq \tilde Y_n$ for every $n$ and it is evident that $(Y_n:n\in\mathbb N)$ has the same properties as $(Y_n:n\in\mathbb N)$.

For every $n\in\mathbb N$ we take an increasing sequence $(F_{nm})_{m=1}^\infty$ of functionally closed sets in $X$ such that $\tilde X_n=\bigcup\limits_{m=1}^\infty F_{nm}$ and denote $C_n=\bigcup\limits_{m=1}^n F_{mn}$. We fix $n\in\mathbb N$ and show that the restriction  $f|_{C_n}:C_n\to Y$ has a continuous extension  $g:X\to Y$.

Since the sets $F_{1n}$,\dots, $F_{nn}$ are disjoint and functionally closed, we may choose open sets  $U_1$,\dots, $U_n$ and a continuous function   $\varphi:X\to [1,n]$ such that  $F_{mn}\subseteq U_m$ and  $\overline{U}_{m}\subseteq\varphi^{-1}(m)$ for every $m\in\{1,\dots,n\}$. Further, in each  $\tilde Y_m$ we take a point $y_m^*$ from Definition~\ref{Def:StarAdhesive}. Since $(Y,\tilde Y_m)\in {\rm Ad}^*(X)$, there exists a continuous map  $g_m:X\to Y$ such that $g_m|_{F_{mn}}=f|_{F_{mn}}$ and $g_m|_{X\setminus U_m}=y_{m}^*$. It implies from the path-connectedness of  $Y$ that there exists a continuous map $\gamma:[1,+\infty)\to Y$ such that $\gamma(m)=y_m^*$ for every $m\in\mathbb N$. For all $x\in X$ we set
$$
g(x)=\left\{\begin{array}{ll}
              g_m(x), & \mbox{if}\,\,\, x\in U_m\,\,\,\mbox{for some}\,\,\,m\in\{1,\dots,n\}, \\
              \gamma(\varphi(x)), & \mbox{otherwise}.
            \end{array}
\right.
$$
Then the map $g:X\to Y$ is continuous and $g|_{C_n}=f|_{C_n}$. Hence, condition (\ref{it:th:char_B1st:2}) holds.

Now we suppose that under conditions (a) or  (b) the implication (\ref{it:th:char_B1st:3})$\Rightarrow$(\ref{it:th:char_B1st:1}) is valid for all ordinals  $\gamma\in [1,\alpha)$ for some  $\alpha\in(1,\omega_1)$ and prove it for $\alpha$. Consider the case $\alpha=\gamma+1<\omega_0$. By Lemma~\ref{amb} for every  $m$ there exists a sequence $(A_{mn})_{n=1}^\infty$ of functionally ambiguous sets of the class $\gamma$ such that $X_m=\mathop{\rm Lim}\limits_{n\to\infty}A_{mn}$. For all $m,n\in\mathbb N$ we set
 \begin{gather}\label{eq:cond_on_Dmn1}
 B_{mn}=A_{mn}\setminus\bigcup\limits_{k<m} A_{kn}.
 \end{gather}
Then each set $B_{mn}$ is functionally ambiguous of the class $\gamma$.

For every $n\in\mathbb N$ we set
$$
g_n(x)=\left\{\begin{array}{ll}
                f_{m}(x), & \mbox{if}\,\,\, x\in B_{mn}\,\,\,\mbox{for}\,\,\,m<n,\\
                f_{n}(x), & \mbox{otherwise}
              \end{array}
\right.
$$
and, applying the inductive assumption, we get  $g_n\in {\rm B}_{\gamma}^{\rm st}(X,Y)$.

It remains to prove that $(g_n)_{n=1}^\infty$ is stably convergent to $f$ on $X$. Fix $x\in X$ and choose a number  $m$ such that  $x\in X_{m}$ and $x\not\in X_k$ for all  $k\ne {m}$. Equality~(\ref{gath:limitset}) implies that there are numbers $N_1,\dots,N_{m}$ such that
 $$
 x\not\in \bigcup\limits_{n\ge N_k} A_{kn}\,\,\mbox{if}\,\, k<m\,\,\,\,\mbox{and}\,\,\,\,  x\in \bigcap\limits_{n\ge N_{m}} A_{mn}.
 $$
Hence, for all $n\ge n_{0}=\max\{N_1,\dots,N_{m}\}$ the inclusion $x\in B_{mn}\cap X_m$ holds. Therefore, $g_n(x)=f_m(x)=f(x)$ for all $n\ge n_0$.

In the case $\alpha\ge \omega_0$ we observe that each set $X_n$ is functionally ambiguous of the class $\alpha+1$ and the sets $A_{mn}$ (together with the sets $B_{mn}$) are functionally ambiguous either of the class $\alpha$, or of classes $<\alpha$ in the case of limit $\alpha$. Then by the inductive assumption we have  $g_n\in {\rm B}_{\gamma}^{\rm st}(X,Y)$, if  $\alpha=\gamma+1>\omega_0$, and  $g_n\in {\rm B}_{<\alpha}^{\rm st}(X,Y)$, if $\alpha$ is limit. In any case, $f\in {\rm B}_{\alpha}^{\rm st}(X,Y)$.
\end{proof}

Let us observe that in the proof of implication   (\ref{it:th:char_B1st:2})$\Rightarrow$(\ref{it:th:char_B1st:1}) we do not use the fact that $Y$ has the functionally closed diagonal. The following example show that this property is essential for implication (\ref{it:th:char_B1st:1})$\Rightarrow$(\ref{it:th:char_B1st:2}).

\begin{exm}
Let $D$ be an uncountable discrete space and $X=Y=D\sqcup\{a \}$ be the Alexandroff compactification of $D$. Then there exists $f\in {\rm B}^{{\rm st}}_1(X,Y)$ such that for every functionally measurable subset $A\subseteq X$ with $a\in A$ the restriction $f|_A$ is discontinuous at $a$.
\end{exm}

\begin{proof} Let $D=\bigsqcup\limits_{n=1}^{\infty}D_n$ such that $|D|=|D_n|$ for every $n\in\mathbb N$. We choose a sequence of bijections $\varphi_n:D_n\to D$ and consider the function $f:X\to Y$,
$$
 f(x)=\left\{\begin{array}{ll}
                         a, & x=a,\\
                         \varphi_n(x), & n\in\mathbb N\,\,\,{\rm and}\,\,\, x\in D_n.
                       \end{array}
 \right.
 $$
Note that $f\in {\rm B}^{\rm st}_1(X,Y)$, because $f$ is the stable limit of the sequence of  continuous functions $f_n:X\to Y$,
$$
 f_n(x)=\left\{\begin{array}{ll}
                         a, & x\in \{a\}\cup(\bigcup\limits_{k=n+1}^\infty D_k),\\
                         \varphi_k(x), & k\leq n\,\,\,{\rm and}\,\,\, x\in D_k.
                       \end{array}
 \right.
 $$

Fix a functionally measurable subset $A\ni a$ of $X$. Since $|D\setminus B|\leq\aleph_0$ for every functionally open or functionally closed subset $B\ni a$ of $X$, we have $|D\setminus A|\leq\aleph_0$. Thus every set $B_n=D_n\setminus A$ is at most countable. Therefore the set $C=\bigcup\limits_{n=1}^\infty \varphi_n(B_n)$ is at most countable too. We choose $d\in D\setminus C$ and consider the neighborhood $V=Y\setminus \{d\}$ of $a$ in $Y$. Since
$$
A\setminus\{x\in A:f(x)\in V\}=\{\varphi_n^{-1}(d):n\in\mathbb N\},
$$
the restriction $f|_A$ is discontinuous at $a$.
\end{proof}

We show in the following example that the properties (a) and (b) in Theorem~\ref{th:char_B1st} are essential.

\begin{exm} {\rm a)   Let $X=[0,1]^2$ and $Y\subseteq X$ be the Sierpi\'{n}ski carpet. Notice that $Y$ is a Peano continuum.
Fix any $x^*\in Y$ and consider the map $f:X\to Y$ such that $f(x)=x$ for $x\in Y$ and $f(x)=x^*$ for $x\in X\setminus Y$. We put $f_1(x)=x$  and $f_2(x)=x^*$ for all $x\in X$. Notice that $X_1=Y$ and $X_2=X\setminus Y$ are ambiguous subsets of the first class in $X$, $X_1\cup X_2=X$ and  $f|_{X_i}=f_i|_{X_i}$ for $i=1,2$.  Therefore, condition~(\ref{it:th:char_B1st:3}) of Theorem~\ref{th:char_B1st} holds. 

Assume  that $f\in {\rm B}_1^{\rm st}(X,Y)$. Take a sequence $(g_n)_{n=1}^\infty$ of continuous functions $g_n:X\to Y$ and  a closed cover $(\tilde X_n:n\in\mathbb N)$ of $X$ such that  $g_n|_{\tilde X_n}=f|_{\tilde X_n}$ and $\tilde X_n\subseteq \tilde X_{n+1}$ for every  $n$.  By the Baire category theorem, there exists $k\in\mathbb N$ such that the set $\tilde X_k\cap Y$ has the nonempty interior in $Y$. Then there exists an open square $L$ in $Y$ such that $\partial L\subseteq \tilde X_k\cap Y$. Since $g_k|_{\partial L}(x)=x$ for all $x\in \partial L$ and $g_k({\rm int}L)\subseteq X\setminus {\rm int}L$, we have that $\partial L$ is a retract of $\overline L$, which is impossible. Hence, condition~(\ref{it:th:char_B1st:2}) of Theorem~\ref{th:char_B1st} does not hold. Consequently, $Y$ fails to be adhesive for $X$.

b) Let $Y=\{0,1\}$ and $f:[0,1]\to Y$ be the characteristic function of the set $\{0\}$. Clearly, $Y$ is a disconnected $\sigma$-adhesive space for $[0,1]$.
Moreover,  the partition $(\{0\},(0,1])$ of $[0,1]$ and the functions $f_1\equiv 1$ and $f_2\equiv 0$ satisfy condition~(\ref{it:th:char_B1st:3}) of Theorem~\ref{th:char_B1st}, but $f$ is not of the first stable Baire class, since each continuous function between $[0,1]$ and $\{0,1\}$ is constant.}
\end{exm}

Clearly,  if  $\alpha=1$, then condition  (\ref{it:th:char_B1st:3}) of Theorem~\ref{th:char_B1st} implies that $f$ is piecewise continuous. It was proved in~\cite[Theorem 6.3]{BaBo} that every piecewise continuous map belongs to the class ${\rm B}_1^{\rm st}(X,Y)$ when $X$ is a normal space and $Y$ is a path-connected space such that $Y\in\sigma {\rm AE}(X)$. The following example shows that a piecewise continuous map need not be of the first stable Baire class even if  $X=\mathbb R^2$  and $Y$ is a contractible subspace of $\mathbb R^2$.

\begin{exm}
 Let $Y=\Delta(C)$ be the cone over the Cantor set $C\subseteq [0,1]$ defined in Example~\ref{exm:cantorAD}, $y^*\in Y$ be a point and $f:[0,1]^2\to Y$ be a map such that
  $$
  f(x)=\left\{\begin{array}{ll}
                x, & x\in Y, \\
                y^*, & \mbox{otherwise}.
              \end{array}
  \right.
  $$
  Then $f$ is piecewise continuous and $f\not\in {\rm B}_1^{\rm st}([0,1]^2,Y)$.
\end{exm}

\begin{proof} Since $f|_Y$ and $f|_{[0,1]^2\setminus Y}$ is continuous and $Y$ is closed in $[0,1]^2$, $f$ is piecewise continuous.

  Assume that there exist a sequence $(f_n)_{n=1}^\infty$ of continuous functions $f_n:[0,1]^2\to Y$ and a closed cover $(X_n:n\in\mathbb N)$ of the square  $[0,1]^2$ such that  $f_n|_{X_n}=f|_{X_n}$ for every  $n$. By the Baire category theorem, there exists $k\in\mathbb N$ such that the set $X_k\cap Y$ has the nonempty interior in $Y$. Let $F$ be a closed square in $[0,1]^2$ such that $F\cap Y\subseteq X_k\cap Y$ and the interior of $F\cap Y$ in  $Y$ is nonempty. Consider the set $B=f_k(F)$. Since  $B$ is a continuous image of $F$, it should be locally connected. On the other hand, $B$ is not a locally connected set, since $B$ is a closed subspace of $Y$ with nonempty interior.
\end{proof}

\section{Monotone functions and their stable Baire measurability}\label{sec:monotone}
We will establish in this section that
$$
\mathcal M\cap{\rm B}_1^{\rm st}(\mathbb R,\mathbb R)=\mathcal M\cap {\rm B}_2^{\rm st}(\mathbb R,\mathbb R)=\dots=\mathcal M\cap {\rm B}_\alpha^{\rm st}(\mathbb R,\mathbb R)=\dots,
$$
where $\mathcal M$ is  the class of all monotone functions.

The following fact immediately follows from definitions and we omit its proof.

\begin{lem}\label{lemma:monotone_dense}
  Let $X\subseteq\mathbb R$, $f:X\to\mathbb R$ be a monotone function and $g:X\to\mathbb R$ be a continuous function such that $f|_D=g|_D$ for some dense set  $D\subseteq X$. Then  $f=g$ on $X$.
\end{lem}

A map $f:X\to Y$ between topological spaces $X$ and $Y$ is said to be {\it weakly discontinuous}, if for any subset $A\subseteq X$ the discontinuity points set of the restriction  $f|_A$ is nowhere dense in $A$. It is easy to see that a map $f$ is weakly discontinuous if and only if the discontinuity points set of the restriction $f|_F$ to any closed set $F\subseteq X$ is nowhere dense in $F$.

\begin{thm}\label{th:mon_stable}
  For a monotone function $f:\mathbb R\to\mathbb R$ the following conditions are equivalent:
  \begin{enumerate}
    \item\label{it:th:mon_stable:1} $f$ is weakly discontinuous;

    \item\label{it:th:mon_stable:2} $f\in {\rm B}_1^{\rm st}(\mathbb R,\mathbb R)$;

    \item\label{it:th:mon_stable:3} $f\in \bigcup\limits_{\alpha<\omega_1}{\rm B}_\alpha^{\rm st}(\mathbb R,\mathbb R)$.
  \end{enumerate}
\end{thm}

\begin{proof}
 The equivalence of (\ref{it:th:mon_stable:1}) and (\ref{it:th:mon_stable:2}) was established in~\cite{BKMM} (see also~\cite{BaBo} for a more general case).

The implication (\ref{it:th:mon_stable:2})$\Rightarrow$(\ref{it:th:mon_stable:3}) is evident.

 We prove that (\ref{it:th:mon_stable:3})$\Rightarrow$(\ref{it:th:mon_stable:1}). Let  $f\in {\rm B}_\alpha^{\rm st}(\mathbb R,\mathbb R)$ for some $\alpha\in[0,\omega_1)$. By Theorem~\ref{th:char_B1st} there exist a sequence of continuous functions $f_n:\mathbb R\to\mathbb R$ and a partition $(X_n:n\in\mathbb N)$ of the real line such that $f_n|_{X_n}=f|_{X_n}$ for every $n\in\mathbb N$. Consider a nonempty closed set  $F\subseteq \mathbb R$. For every $n$ we denote $G_n={\rm int}_F\overline{X_n\cap F}$. Since $F$ is a Baire space, the union $G=\bigcup\limits_{n=1}^\infty G_n$ is dense in $F$. The equality   $f|_{X_n\cap F}=f_n|_{X_n\cap F}$ and Lemma~\ref{lemma:monotone_dense} imply that $f|_{G_n}=f_n|_{G_n}$ for every $n$. Since every function   $f_n$ is continuous and the set $G_n$ is open in $F$, we have $G_n\subseteq C(f|_F)$. Hence, $G\subseteq C(f|_F)$, which implies that $f$ is weakly discontinuous.
\end{proof}

As a corollary of Theorem~\ref{th:mon_stable} we obtain the following result.
\begin{prop}
  There exists a function $f\in {\rm B}_1(\mathbb R,\mathbb R)\setminus \bigcup\limits_{\alpha<\omega_1}{\rm B}_\alpha^{\rm st}(\mathbb R,\mathbb R)$.
\end{prop}

\begin{proof}
We consider the increasing function   $f:\mathbb R\to\mathbb R$,
  $$
  f(x)=\sum\limits_{r_n\le x}\frac{1}{2^n},
  $$
where $\mathbb Q=\{r_n:n\in\mathbb N\}$. Since $f$ is monotone, $f\in{\rm B}_1(\mathbb R,\mathbb R)$. But  $D(f)=\mathbb Q$. Therefore, $f$ is not weakly discontinuous. Hence,  $f\not\in \bigcup\limits_{\alpha<\omega_1}{\rm B}_\alpha^{\rm st}(\mathbb R,\mathbb R)$ by Theorem~\ref{th:mon_stable}.
\end{proof}


\begin{thebibliography}{99}
\bibitem{BaBo} T. Banakh, B. Bokalo {\it On scatteredly continuous maps between topological spaces},  Topology Appl. {\bf 157} (1) (2010), 108–-122.

\bibitem{BKMM} T. Banakh, S. Kutsak, O. Maslyuchenko, V. Maslyuchenko, {\it Direct and inverse problems of the Baire classifications of integrals dependent on a parameter}, Ukr. Math. J. {\bf 56} (11) (2004), 1443--1457.

\bibitem{Bru} A.~Bruckner, {\it Differentiation of Real Functions} [2nd
ed.], Providence, RI: American Mathematical Society, 1994,
195~p.


\bibitem{CsLacz:1}  \'{A}. Cs\'{a}sz\'{a}r, M. Laczkovich, {\it Discrete and equal convergence}, Studia Sci. Math. Hungar. {\bf 10}
(1975), 463--472.

\bibitem{CsLacz:2}  \'{A}. Cs\'{a}sz\'{a}r, M. Laczkovich, {\it Some remarks on discrete Baire classes}, Acta Math. Acad. Sci. Hungar. {\bf 33} (1979), 51--70.

\bibitem{ChRos} F. Chaatit, H. Rosenthal, {\it On differences of semi-continuous functions}, Quaest. Math. {\bf 23} (3) (2000) 295--311.

\bibitem{HOR} R. Haydon, E. Odell, H. Rosenthal, {\it On certain classes of Baire-1 functions with applications to Banach space theory}, in: Functional Analysis, Austin, TX, 1987/1989, in: Lecture Notes in Math., vol. 1470, Springer, Berlin, 1991, pp. 1--35.

\bibitem{Karlova:Mykhaylyuk:Comp} O. Karlova, V. Mykhaylyuk, {\it On composition of Baire functions}, arxiv.org/abs/1511.08982.

\bibitem{Kuratowski:Top:1} K.~Kuratowski, {\it Topology. Volume 1}, Academic Press (1966).

\bibitem{KuSerp} C. Kuratowski and W. Sierpinski, {\it Les fonctions de classe  et les ensembles connexes ponctiformes}, Fund. Math. {3} (1922) 303--313.



\end{thebibliography}
\end{document}